\documentclass[12pt]{amsart}
\usepackage{enumerate}
\usepackage[hyperindex,breaklinks,colorlinks,citecolor=blue,pagebackref]{hyperref}
\usepackage[latin1]{inputenc}
\usepackage{amsmath, amsthm, amsfonts, amssymb}
\usepackage{mathrsfs}
\usepackage{graphicx}
\usepackage{latexsym}
\usepackage{fullpage}
\usepackage[all]{xy}
\usepackage{color}
\usepackage{stmaryrd}
\usepackage[mathscr]{euscript}

\newtheorem{theorem}{Theorem}[section]
\newtheorem{lemma}[theorem]{Lemma}
\newtheorem{proposition}[theorem]{Proposition}
\newtheorem{corollary}[theorem]{Corollary}

\newtheorem{definition}[theorem]{Definition}

\theoremstyle{remark}
\newtheorem{remark}{Remark}[section]
\newtheorem{rmks}{Remarks}[section]
\newtheorem*{Examples}{Examples}

\newcommand{\C}{\mathcal{C}}

\newcommand{\F}{\mathcal{F}}

\renewcommand{\P}{\mathscr{P}}

\newcommand{\U}{\mathcal{U}}

\newcommand{\X}{\mathcal{X}}

\newcommand{\MLR}{\mathsf{MLR}}

\newcommand{\dom}{\text{dom}}

\newcommand{\Tree}{\mathsf{Tree}}

\newcommand{\llb}{\llbracket}
\newcommand{\rrb}{\rrbracket}

\newcommand{\Supp}{\mathsf{Supp}}

\definecolor{purple}{rgb}{.9,0.2,.9}

\newcommand{\eqdef}{\mathrel{\mathop:}=}

\newcommand{\cs}{2^\omega}
\newcommand{\ccs}{3^\omega}
\newcommand{\uh}{{\upharpoonright}}
\renewcommand{\phi}{\varphi}
\newcommand{\str}{2^{<\omega}}

\newcommand{\diverge}{{\uparrow}}

\newcommand{\ran}{\mathrm{ran}}

\usepackage{xcolor}	
	\usepackage{soul}

\title{The Interplay of Classes of Algorithmically Random Objects}
\author{Quinn Culver and Christopher P.\ Porter}
\date{} % delete this line to display the current date

\thanks{Christopher Porter was supported by the National Science Foundation under grant OISE-1159158 as part of the International Research Fellowship Program.  Porter was also funded by the John Templeton Foundation (`Structure and Randomness in the Theory of Computation' project). The opinions expressed in this publication are those of the authors and do not necessarily reflect the views of the John Templeton Foundation.}

\begin{document}

\maketitle

\begin{abstract}
We study algorithmically random closed subsets of $\cs$, algorithmically random continuous functions from $\cs$ to $\cs$, and algorithmically random Borel probability measures on $\cs$, especially the interplay between these three classes of objects. Our main tools are preservation of randomness and its converse, the no randomness ex nihilo principle, which say together that given an almost-everywhere defined computable map between an effectively compact probability space and an effective Polish space, a real is Martin-Löf random for the pushforward measure if and only if its preimage is random with respect to the measure on the domain. These tools allow us to prove new facts, some of which answer previously open questions, and reprove some known results more simply.

Our main results are the following.  First we answer an open question in \cite{Barmpalias:2008aa} by showing that $\X\subseteq\cs$ is a random closed set if and only if it is the set of zeros of a random continuous function on $\cs$.  As a corollary we obtain the result that the collection of random continuous functions on $\cs$ is not closed under composition.  Next, we construct a computable measure $Q$ on the space of measures on $\cs$ such that $\X\subseteq\cs$ is a random closed set if and only if $\X$ is the support of a $Q$-random measure.  We also establish a correspondence between random closed sets and the random measures studied in \cite{Culver:2014aa}.  Lastly, we study the ranges of random continuous functions, showing that the Lebesgue measure of the range of a random continuous function is always contained in $(0,1)$.  
\end{abstract}

\section{Introduction}

In this paper, we have two primary goals: (1) to study the interplay between algorithmically random closed
  sets on $\cs$, algorithmically random continuous functions on $\cs$,
  and algorithmically random measures on $\cs$; and (2) to apply two central results, namely the preservation of
  randomness principle and the no randomness ex nihilo principle, to
  the study of the algorithmically random objects listed above.

Barmpalias, Brodhead, Cenzer, Dashti and Weber initiated the study of
algorithmically random closed subsets of $\cs$ in
\cite{Barmpalias:2007aa}.  Algorithmically random closed sets were
further studied in, for instance, \cite{Axon:2010aa},
\cite{Diamondstone:2012aa}, and \cite{Cenzer:2013aa}.  In the spirit
of their definition of algorithmically random closed set, Barmpalias,
Brodhead, Cenzer, Dashti and Weber also defined a notion of
algorithmically random continuous function on $\cs$ in
\cite{Barmpalias:2008aa}.  The connection between random closed sets
and effective capacities was explored in \cite{Brodhead:2011aa}.  More
recently, Culver has studied algorithmically random measures on $\cs$
in \cite{Culver:2014aa}.

One of the central results  \cite{Barmpalias:2008aa} is that the set of 
zeros of a random continuous function of $\cs$ is a random closed 
subset of $\cs$. Inspired by this result, we here investigate similar bridge
results, which allow us to transfer information about one class of
algorithmically random objects to another.

Two tools that are central to our investigation, mentioned in (2)
above, are the preservation of randomness principle and the no
randomness ex nihilo principle.  In $\cs$, the space of infinite
binary sequences, the preservation of randomness principle tells us
that if $\Phi:\cs\rightarrow\cs$ is an effective map and $\mu$ is a
computable measure on $\cs$, then $\Phi$ maps $\mu$-random members of
$\cs$ to members of $\cs$ that are random with respect to the measure
$\nu$ obtained by pushing $\mu$ forward via $\Phi$.  Furthermore, the
no randomness ex nihilo principle tells us that any sequence that is
random with respect to $\nu$ is the image of some $\mu$-random
sequence under $\Phi$.  Used in tandem, these two principles allow us
to conclude that the image of the $\mu$-random sequences under $\Phi$
are precisely the $\nu$-random sequences.

With the exception of \cite{Culver:2014aa}, the studies listed above
do not make use of these two tools used in tandem.  As we will show,
they not only allow for the simplification of a number of proofs in
the above-listed studies, but they also allow us to answer a number of
questions that were left open in these studies.

The outline of the remainder of this paper is as follows.  In Section
\ref{sec-background}, we provide the requisite background for the rest
of the paper.  In Section \ref{sec-aro}, we review the basics of
algorithmic randomness, including preservation and the no randomness
ex nihilo principle.  We also provide the definitions of algorithmic
randomness for closed sets of $\cs$, random continuous functions on
$\cs$, and measures on $\cs$ and list some basic properties of these
objects.  Section \ref{sec-prelim} contains simplified proofs of some
previously obtained results from \cite{Barmpalias:2007aa} and
\cite{Barmpalias:2008aa}, as well as a proof of a conjecture in
\cite{Barmpalias:2008aa} that every random closed subset of $\cs$ is
the set of zeros of a random continuous function on $\cs$.  We study
the support of a certain class of random measures in Section
\ref{sec-support} and establish a correspondence between between
random closed sets and the random measures studied in
\cite{Culver:2014aa}.  Lastly, in Section \ref{sec-range}, we prove
that the Lebesgue measure of the range of a random continuous function
on $\cs$ is always non-zero, from which it follows that no random
continuous function is injective (which had not been previously
established).  We also strengthen a result in \cite{Barmpalias:2008aa}
(namely, that not every random continuous function is surjective) by
proving that no random continuous function is surjective, from which
it follows that the Lebesgue measure of the range of a random
continuous function is never equal to one.

\section{Background}\label{sec-background}

\subsection{Some topological and measure-theoretic basics}

For $n=\{0,1,\ldots n-1\}\in \omega$, the set of all finite strings
over the alphabet $n$ is denoted $n^{<\omega}$. When $n=2$, we let
$\sigma_{0}, \sigma_{1}, \sigma_{2},\dotsc$ be the canonical
length-lexicographic enumeration of $2^{<\omega}$, so that
$\sigma_{0}=\epsilon$ (the empty string), $\sigma_{1}=0$, $\sigma_{2}=1$, etc.

$n^{\omega}$ is the space of all infinite sequences
over the alphabet $n$. The elements of $n^{\omega}$ are also called
\emph{reals}. The product topology on $n^{\omega}$ is generated by the
clopen sets 
\[
\llb\sigma\rrb= \{x \in n^{\omega}: x\succ \sigma\},
\] where
$\sigma \in n^{<\omega}$ and $x \succ \sigma$ means that $\sigma$ is
an initial segment of $x$. When $x$ is a real and $k\in\omega$, $x\restriction k$ 
denotes the initial segment of $x$ of length $k$.

For $\sigma,\tau\in n^{<\omega}$, $\sigma^{\frown}\tau$ denotes the concatenation
of $\sigma$ and $\tau$.  In some cases, we will write this concatenation
as $\sigma\tau$.

A \emph{tree} is a subset of $n^{<\omega}$ that is closed under
initial segments; i.e.\ $T\subseteq n^{<\omega}$ is a tree if $\sigma
\in T$ whenever $\tau \in T$ and $\sigma \preceq \tau$. A \emph{path}
through a tree $T\subseteq n^{<\omega}$ is a real $x \in n^{\omega}$
satisfying $x \restriction k \in T$ for every $k$. The set of all
paths through a tree $T$ is denoted $[T]$. Recall the correspondence
between closed sets and trees.

\begin{proposition}
  A set $C\subseteq n^{\omega}$ is closed if and only if $C=[T]$ for
  some tree $T\subseteq n^{<\omega}$. Moreover, $C$ is non-empty if and
  only if $T$ is infinite.  \label{prop:closedpathstree}
\end{proposition}

A \emph{measure} $\mu$ on $n^{\omega}$ is a function that assigns to
each Borel subset of $n^{\omega}$ a number in the unit interval
$[0,1]$ and satisfies $\mu(\bigcup_{i \in \omega} B_{i})= \sum_{i \in
  \omega} \mu(B_{i})$ whenever the $B_{i}$'s are pairwise disjoint.
Carath\'{e}odory's extension theorem guarantees that the conditions
\begin{itemize}
  \item $\mu(\llb\epsilon\rrb)= 1$ and
  \item $\mu(\llb\sigma\rrb) = \mu(\llb\sigma0\rrb) +
  \mu(\llb\sigma1\rrb) + \ldots + \mu(\llb\sigma^{\frown}(n-1)\rrb) $ for
  all $\sigma \in n^{<\omega}$
\end{itemize}
uniquely determine a measure on $n^{\omega}$.  Thus a measure is
identified with a function $\mu\colon n^{<\omega} \to [0,1]$
satisfying the above conditions and $\mu(\sigma)$ is often written
instead of $\mu(\llb\sigma\rrb)$.  The Lebesgue measure $\lambda$
on $n^\omega$ is defined by $\lambda(\sigma) = n^{-|\sigma|}$ for each string
$\sigma\in n^{<\omega}$.

Given a measure $\mu$ on $n^\omega$ and $\sigma,\tau\in n^{<\omega}$, 
$\mu(\sigma\tau\mid\sigma)$ is defined to be
\[
\mu(\sigma\tau\mid\sigma)=\dfrac{\mu(\llb\sigma\tau\rrb)}{\mu(\llb\sigma\rrb)}.
\]

\subsection{Some computability theory}
\label{sec:some-comp-theory}

We assume the reader is familiar with the basic concepts of computability
theory as found, for instance, in the early chapters of \cite{Soare:1987aa}.

A $\Sigma^0_1$ \emph{class} $S\subseteq n^\omega$ is an effectively 
open set, i.e., an effective union of basic clopen subsets of $n^\omega$.
$P\subseteq n^\omega$ is a $\Pi^0_1$ \emph{class} if $\cs\setminus P$ is a 
$\Sigma^0_1$ class.

A partial function $\Phi\colon \subseteq n^{\omega} \to m^{\omega}$ is
\emph{computable} if the preimage of a $\Sigma^{0}_{1}$ subset of
$m^{\omega}$ is a $\Sigma^{0}_{1}$ subset of the domain of $\Phi$,
uniformly; that is, if for every $\Sigma^{0}_{1}$ class $U\subseteq
m^{\omega}$, there is a $\Sigma^{0}_{1}$ class $V\subseteq n^{\omega}$
such that $\Phi^{-1}(U) = V \cap \dom(\Phi)$, and an index for $V$ can be
uniformly computed from an index for $U$. Equivalently, $\Phi\colon
\subseteq n^{\omega} \to m^{\omega}$ is computable if there is an
oracle Turing machine that when given $x \in n^{\omega}$ (as an
oracle) and $k \in \omega$ outputs $\Phi(x)(k)$.  We can relativize the notion 
of a computable function $\Phi\colon \subseteq n^{\omega} \to m^{\omega}$
to any oracle $z\in\cs$ to obtain a $z$-\emph{computable} function.

A measure $\mu$ on $n^\omega$ is computable if $\mu(\sigma)$ 
is a computable real number, uniformly in $\sigma\in n^{<\omega}$. 
Clearly, the Lebesgue measure $\lambda$ is computable.

If $\mu$ is a computable measure on $n^\omega$ and 
$\Phi\colon \subseteq n^{\omega} \to m^{\omega}$ is a computable function
defined on a set of $\mu$-measure one, then the \emph{pushforward measure}
$\mu_\Phi$ defined by
\[
\mu_\Phi(\sigma)=\mu(\Phi^{-1}(\sigma))
\]
for each $\sigma\in m^{<\omega}$ is a computable measure.

\section{Algorithmically random objects}\label{sec-aro}

In this section, we lay out the definitions of the various algorithmically random objects that are the subject of this study.
For more details, see \cite{Nies:2009aa} or \cite{Downey:2010aa}.

\subsection{Algorithmically random sequences}  

\begin{definition}
  Let $\mu$ be a computable measure on $n^{\omega}$ and let $z\in
  m^{\omega}$. 
  \begin{itemize}
   \item[(i)] A \emph{$\mu$-Martin-L\"of test relative to $z$} (or simply a \emph{$\mu$-test relative to $z$}) is a uniformly $\Sigma^{0,z}_{1}$ sequence 
   $(U_{i})_{i\in\omega}$ of subsets  of $n^{\omega}$  with $\mu (U_{n}) \leq 2^{-n}$.
  \item[(ii)] $x\in n^\omega$ passes such a test $(U_{i})_{i\in\omega}$ if $x\notin \bigcap_{n} U_{n}$.
  \item[(iii)] $x\in n^\omega$ is $\mu$-Martin-L\"of random relative to $z$ if $x$ passes every $\mu$-Martin-L\"of test relative to $z$.
  \end{itemize}
\end{definition}
  
  We will often write ``random" instead of ``Martin-L\"of random," since we are only working with one notion of randomness in
  this paper (although there are other reasonable notions of algorithmic randomness that one might consider).   The collection 
  of $\mu$-random sequences relative to $z$ will be denoted $\MLR^{z}_{\mu}$.  When $z$ is computable we 
  simply write $\MLR_{\mu}$, say that $x$ is $\mu$-random, call $(U_{i})_{i\in\omega}$ a $\mu$-test, etc.

The following is well-known and straightforward.
\begin{proposition}\label{prop-pi01-null}
  Let $\mu$ be a computable measure on $n^{\omega}$ and $z\in
  m^{\omega}$. If $C\subseteq n^{\omega}$ is $\Pi^{0,z}_{1}$ and $\mu
  (C)=0$, then $C \cap \MLR^{z}_{\mu}=\emptyset$.
\end{proposition}

% \begin{itemize}
%   \item Martin-L\"of randomness on $\cs$ and more general spaces
% \end{itemize}

The following is likely folklore, but was at least observed
in~\cite{Bienvenu:2012aa}.  
\begin{proposition}
  Let $\mu$ be a computable measure on $n^{\omega}$. If $\Phi\colon
  \subseteq n^{\omega} \to m^{\omega}$ is computable with $\mu(\dom
  (\Phi))=1$, then $\MLR_{\mu}\subseteq \dom(\Phi)$.
\end{proposition}

\begin{lemma}[Folklore]\label{lem-total}
  Let $\Phi\colon\subseteq \cs \to \cs$ be computable and $C$ a
  $\Pi^{0}_{1}$ subset of $\dom(\Phi)$. Then $\Phi(C)\in \Pi^{0}_{1}$,
  uniformly. \label{lem:CompImPi01}
\end{lemma}

One of the central tools that we will use in this study is the following.  

\begin{theorem}[Preservation of Randomness \cite{Levin:1970aa} and No Randomness Ex Nihilo \cite{Shen:2008aa}.]
  Let $\Phi\colon \subseteq \cs \to \cs$ be computable with $\lambda
  (\dom(\Phi))=1$.
  \begin{enumerate}[(i)]
    \item If $x \in \MLR_{\lambda}$ then $\Phi(x) \in \MLR_{\lambda_\Phi}$.
    \item If $y \in \MLR_{\lambda_\Phi}$, then there is $x\in
    \MLR_{\lambda}$ such that $\Phi(x)=y$.
  \end{enumerate}
  \label{thm:PoR-and-NReN}
\end{theorem}

\begin{proof}\hfill
  \begin{enumerate}[(i)]
    \item If $\Phi(x) \notin \MLR_{\lambda_\Phi}$, then $\Phi(x) \in
    \bigcap_{n} V_{n}$ for some $\lambda_\Phi$-test $(V_i)_{i\in\omega}$. But then
    $x \in \bigcap_{n} \Phi^{-1}(V_{n})$ and $\lambda(\Phi^{-1}(V_{n})) \leq
    2^{-n}$. Moreover, because $\Phi$ is computable (on its domain),
    $\Phi^{-1}(V_{n})=U_{n}\cap \dom(\Phi)$ for some $\Sigma^{0}_{1}$ class
    $U_{n}$. Since $\lambda(\dom(\Phi))=1$, $\lambda(\Phi^{-1}(U_{n}))\leq
    2^{-n}$. Thus $x \notin \MLR_{\lambda}$.

    \item Let $(U_{i})_{i\in\omega}$ be a universal test for $\lambda$-randomness and
    set $K_{n}=X\setminus U_{n}$. Then $\Phi(K_{n})$ is uniformly $\Pi^{0}_{1}$ by
    Lemma~\ref{lem:CompImPi01}, so $Y\setminus \Phi(K_{n})$ is uniformly
    $\Sigma^{0}_{1}$. Because $\lambda_\Phi(Y\setminus \Phi(K_{n})) = 1-
    \lambda_\Phi(\Phi(K_{n})) \leq 1- \lambda(K_{n}) \leq
    2^{-n}$, the sets $Y\setminus\Phi(K_{n})$ form a test for $\lambda_\Phi$-randomness. So if $y \in \MLR_{\lambda_\Phi}$, then
    $y \notin Y\setminus\Phi(K_{n})$ for some $n$; i.e.\ $y \in \Phi(K_{n})$. The
    proof is now complete since $K_{n} \subseteq \MLR_{\lambda}$.
  \end{enumerate}
\end{proof}

We will also use a relativization of Theorem~\ref{thm:PoR-and-NReN}.
\begin{corollary}
  Let $\Phi\colon \subseteq \cs \to \cs$ be computable relative to $z\in
  \cs$ with $\lambda (\dom(\Phi))=1$.
  \begin{enumerate}[(i)]
    \item If $x \in \MLR^{z}_{\lambda}$ then $\Phi(x) \in
    \MLR_{\lambda_\Phi}^{z}$.
    \item If $y \in \MLR_{\lambda_\Phi}^{z}$, then there is
    $x\in \MLR_{\lambda}^{z}$ such that $\Phi(x)=y$.
  \end{enumerate}
\end{corollary}

Lastly, the following result, known as van Lambalgen's theorem, will be useful to us.

\begin{theorem}[\cite{VanLambalgen:1990aa}]\label{thm-vlt}
Let $\mu$ and $\nu$ be computable measures on $m^\omega$ and $n^\omega$, respectively.  Then for $(x,y)\in m^\omega\times n^\omega$,
$(x,y)\in\MLR_{\mu\otimes\nu}$ if and only if $x\in\MLR_\mu^y$ and $y\in\MLR_\nu$.
\end{theorem}

\subsection{Algorithmically random closed subsets of $\cs$}

% $\mathcal{C}(\cs)$: the space of closed subsets of $\cs$

% \begin{definition}
%   Random closed sets (natural definition)
% \end{definition}

Let $\C(\cs)$ denote the collection of all non-empty closed subsets of
$\cs$. As noted in Proposition~\ref{prop:closedpathstree}, these are
the sets of paths through infinite binary trees. Thus to randomly
generate a non-empty closed set, it suffices to randomly generate an
infinite tree. Following \cite{Barmpalias:2007aa}, we will code infinite trees by reals in $3^{\omega}$, thereby reducing 
the process of randomly generating infinite trees to the process of randomly generating reals.

Given $x\in 3^{\omega}$, define a tree $T_{x}\subseteq \str$
inductively as follows. First $\epsilon$, the empty string is
automatically in $T_{x}$. Now suppose $\sigma\in T_{x}$ is the $(i+1)$-st extendible node in $T_x$. Then
\begin{itemize}
  \item ${\sigma}^{\frown} 0 \in T_{x}$ and ${\sigma}^{\frown} 1
  \notin T_{x}$ if $x(i)=0$;
  \item ${\sigma}^{\frown} 0 \notin T_{x}$ and ${\sigma}^{\frown}
  1 \in T_{x}$ if $x(i)=1$;
  \item ${\sigma}^{\frown} 0 \in T_{x}$ and ${\sigma}^{\frown} 1
  \in T_{x}$ if $x(i)=2$.
\end{itemize}
Under this coding $T_{x}$ has no dead ends and hence is always
infinite. Note that every tree without dead ends can be coded by some
$x\in\cs$.

\begin{definition}\label{def:closed-set-defn}
  A non-empty closed set $C \in \C(\cs)$ is a \emph{random closed set}
  if $C=[T_{x}]$ for some $x \in \MLR_{\lambda}$.
\end{definition}

The main facts about random closed sets that we will use in the sequel are 
as follows.

\begin{theorem}[\cite{Barmpalias:2007aa}]\label{thm-racs-null}
Every random closed set has Lebesgue measure zero.
\end{theorem}

\begin{theorem}[\cite{Barmpalias:2007aa}]\label{thm-racs-perfect}
Every random closed set is perfect.
\end{theorem}

% \begin{theorem}[Axon]
%   Equivalence of two definitions of random closed sets
% \end{theorem}

% Relevant facts about random closed sets $\mu$-random closed sets
% Members of random closed sets: every closed set that is random with
% respect to a computable measure on $\C(\cs)$ contains a sequence that
% is random with respect to some computable measure on $\cs$.

% \chris{We have a proof of this using preservation of randomness--the
%   leftmost path is random wrt a computable measure}

% \begin{theorem}
%   \begin{itemize}
%     \item[(i)] Every random closed set contains a Martin-L\"of random
%     sequence.
%     \item[(ii)] Every Martin-L\"of random sequence is contained in
%     some random closed set.
%   \end{itemize}
% \end{theorem}

% \begin{proposition}
%   ...
% \end{proposition}

\subsection{Algorithmically random continuous functions on $\cs$}

Let $\F(\cs)$ denote the collection of all continuous $F\colon
\subseteq \cs \to \cs$. To define a random continuous function, we
code each element of $\F(\cs)$ by a real $x\in 3^{\omega}$ (as carried out in \cite{Barmpalias:2008aa}). 
The coding is a labeling of the edges of $2^{\omega}$ (or equivalently, all nodes in $\str$ 
except  $\epsilon$) by the digits of $x$. Having labeled the edges according to 
$x$, the function $F_{x}$ coded by $x$ is defined by $F_{x}(y)=z$ if $z$ is the 
element of $2^{\omega}$ left over after following $y$ through the labeled tree
and removing the $2$'s. (In the case where only finitely many $0$'s
and $1$'s remain after removing the $2$'s, $F_{x}(y)$ is undefined.)

Formally, define a labeling function $\ell_{x}\colon
2^{<\omega} \setminus \{\epsilon\} \to 3$ by $\ell_{x}(\sigma_{i}) =
x_{i-1}$ (recall that $(\sigma_i)_{i\in\omega}$ is the standard
enumeration of $\str$). Now $F_{x}\in \F(\cs)$ is defined by $F_{x}(y) = z$ if and
only if $z$ is the result of removing the $2$'s from the sequence
$\ell_{x}(y\restriction 1), \ell_{x}(y\restriction 2),
\ell_{x}(y\restriction 3), \dotsc$.  Note that every $F\in\F(\cs)$ has 
infinitely many codes.

\begin{definition}\label{def:RCF-defn}
  A function $F\in \F(\cs)$ is a \emph{random continuous function} if
  $F=F_{x}$ for some $x \in \MLR_{\lambda}$.
\end{definition}

\begin{remark}\label{rmk-rcfs}
  $F_{x}$ is continuous (on its domain) because it is computable
  relative to some oracle, namely $x$.  Since $\cs$ is compact and
  Hausdorff, it follows that $F_x$ is a closed map and hence that
  $\ran(F)$ is $\Pi^{0,F}_1$.
\end{remark}

We will make use of the following facts about random continuous 
functions.

\begin{theorem}[\cite{Barmpalias:2008aa}]\label{thm-rcfs-image-of-comp}
  If $F\in\F(\cs)$ is random and $x\in\cs$ is computable, then
  $F(x)\in\cs$ is random.
\end{theorem}

\begin{theorem}[\cite{Barmpalias:2008aa}]\label{thm-rcfs-total}
  If $F\in\F(\cs)$ is random, then $F$ is total.
\end{theorem}

\subsection{Algorithmically random measures on
  $\cs$} \label{subsec-random-measures} Let $\mathscr{P}(\cs)$ be the
space of probability measures on $\cs$.  Given $x\in\cs$, the
$n$-th column $x_n$ of $x$ is defined by $x_n(k)=1$ if and
only if $x(\langle n, k\rangle) = 1$, where $\langle n, k\rangle$ is
some fixed computable bijection between $\omega^2$ and $\omega$. We
write $x = \oplus_{n\in\omega}x_n$. We define a map $\Psi:\cs\rightarrow\mathscr{P}(\cs)$ that
sends a real $x$ to the measure $\mu_x$ satisfying (i)
$\mu_x(\epsilon)=1$ and (ii) $\mu_x(\sigma_n0)=x_n \cdot
\mu_x(\sigma_n)$, where $x_n$ is the real number corresponding to the
$n$-th column of $x$ and $\sigma_n$ is the $n$-th element 
in the standard enumeration of $\cs$.  This coding was first given in \cite{Culver:2014aa}.

\begin{definition}\label{defn-random-measure}
  A measure $\mu\in \mathscr{P}(\cs)$ is a \emph{random measure} if
  $\mu=\mu_{x}$ for some $x \in \MLR_{\lambda}$.
\end{definition}

Let $P$ be the pushforward measure on $\mathscr{P}(\cs)$ induced by
$\lambda$ and $\Psi$.  Then we have the following.

\begin{theorem}\cite{Culver:2014aa}
  Let $\nu\in\mathscr{P}(\cs)$.  Then $\nu\in\MLR_P$ if and only if
  $\nu=\mu_x$ for some $x\in\MLR_\lambda$.
\end{theorem}

The support of a measure $\mu$ on $\cs$ is defined to be
\[
\Supp(\mu)=\{x\in\cs:(\forall n)[\mu(x\uh n)>0]\}
\]
It is not hard to see that $\Supp(\mu)=\cs$ for every random measure.

In \cite{Culver:2014aa}, it was shown, among other results, that
random measures are atomless (that is, $\mu(\{x\})=0$ for every $x\in\cs$) 
and that the reals that are random with respect to some random measure are 
precisely the reals in
$\MLR_{\lambda}$.

% Note that a more general class of random measures can be obtained by
% considering a product measure on $(\cs)^\omega$.  Since the map
% $(X_i)_{i\in\omega}\mapsto\bigoplus_{i\in\omega}X_i$ is a
% homeomorphism, $(\cs)^\omega$ and $\cs$ are isomorphic as measure
% spaces.  So if we specify a product measure
% $\bigotimes_{i\in\omega}\mu_i$ on $(\cs)^\omega$, this will induce a
% map from $(\cs)^\omega$ to $\mathscr{P}(\cs)$ that sends a
% ($\bigotimes_{i\in\omega}\mu_i$)-random tuple
% $(X_i)_{i\in\omega}\in(\cs)^\omega$ to a measure that is random for
% the push-forward measure.  For example, a random measure as defined
% above is obtained as the image of a
% ($\bigotimes_{i\in\omega}\lambda$)-random tuple
% $(X_i)_{i\in\omega}\in(\cs)^\omega$.

\section{Applications of Randomness Preservation and No Randomness Ex Nihilo}\label{sec-prelim}

In this section, we demonstrate the usefulness of preservation of
randomness and the no randomness ex nihilo principle in the study of
algorithmically random objects such as closed sets, continuous
functions, and so on.

As a warm-up, we provide a new, simpler proof of a known result from
\cite{Barmpalias:2007aa}.

\begin{theorem}\label{thm-members-RCS}
  Every random closed set contains an element of $\MLR_{\lambda}$ and
  every element of $\MLR_{\lambda}$ is contained in some random closed
  set.
\end{theorem}

\begin{proof}
  We define a computable map $\Phi\colon \C(\cs) \times \cs \to \cs$ that
  pushes forward the product measure $\lambda_{\C} \otimes \lambda$ to
  $\lambda$ and satisfies $\Phi(C,x)\in C$ for every pair $(C,x)\in
  \C(\cs) \times \cs$. Having done this, preservation of randomness
  and no randomness ex nihilo imply that the image of a
  $(\lambda_{\C}\otimes \lambda)$-random pair is $\lambda$-random and
  any $\lambda$-random is the image of some $(\lambda_{\C} \otimes
  \lambda)$-random pair. The result then follows because by Van
  Lambalgen's theorem (Theorem \ref{thm-vlt}), a pair $(C,x)$ is $(\lambda_{\C} \otimes
  \lambda)$-random if and only if $C$ is $\lambda_{\C}$-random and $x$
  is $\lambda$-random relative to $C$.

  The map $\Phi$ provides a path through $C$ (when viewed as the paths through a tree) by using $x$ to tell us which way to go 
  whenever we encounter a branching node. Specifically,
  having $\Phi(C,x)\restriction n = \sigma$ such that $\llb\sigma\rrb \cap C
  \neq \emptyset$, we define $\Phi(C,x)(n) =0$ if $\llb\sigma 1\rrb \cap C =
  \emptyset$ and $\Phi(C,x)(n)=1$ if $\llb\sigma 0\rrb \cap C = \emptyset$. If
  neither $\llb\sigma 0\rrb \cap C = \emptyset$ nor $\llb\sigma 1\rrb \cap C =
  \emptyset$, then we define $\Phi(C,x)(n) = x(n)$.

  The map $\Phi$ is clearly computable. It pushes $\lambda_{\C} \otimes
  \lambda$ forward to $\lambda$ because if $\Phi$ has output $\sigma \in
  2^{n}$ after $n$ steps, then $\Phi$ outputs a next bit of $0$ if and only if either
  $\llb\sigma 1\rrb \cap C = \emptyset$ or both $\llb\sigma 1\rrb\cap C$ and
$\llb\sigma 0\rrb \cap C$ are non-empty and $x(n)=0$. The former happens
  with probability $\frac{1}{3}$, and the latter happens with
  probability $\frac{1}{6}=\frac{1}{3}\cdot \frac{1}{2}$ by independence. The
  proof is now complete since $\frac{1}{3} + \frac{1}{6} =
  \frac{1}{2}$.
\end{proof}

Let $F\in\F(\cs)$.  We define the zeros of $F$ to be $Z_{F}=\{x:F(x)=0^\omega\}$, which is clearly a
closed subset of $\cs$.  In \cite{Barmpalias:2008aa}, the following
was shown.

\begin{theorem}[\cite{Barmpalias:2008aa}]\label{thm:zeros}
  Let $F\in\F(\cs)$ be random.  Then $Z_{F}$ is a random closed set
  provided it is non-empty.
\end{theorem}
In \cite{Barmpalias:2008aa} it was conjectured that the converse also
holds, but this was left open.  We show that this is the case.  To do
so, we provide a new proof of Theorem \ref{thm:zeros}, from which the
converse follows immediately.  We also make use of an alternative
characterization of random closed sets due to Diamondstone and
Kj{\o}s-Hanssen \cite{Diamondstone:2012aa}.

Just as a binary tree with no dead ends is coded by a sequence in
$3^{\omega}$ (see the paragraph preceding
Definition~\ref{def:closed-set-defn}), an arbitrary binary tree is
coded by a sequence in $4^{\omega}$, except now a $3$ at a node
indicates that the tree is dead above that node. That is, given $x\in
4^{\omega}$, we define a tree $S_{x}\subseteq \str$ inductively as
follows. First $\epsilon$, the empty string, is  included in
$S_{x}$ by default. Now suppose that $\sigma\in S_{x}$ is the $(i+1)$-st
extendible node in $S_x$. Then
\begin{itemize}
  \item ${\sigma}^{\frown} 0 \in S_{x}$ and ${\sigma}^{\frown} 1
  \notin S_{x}$ if $x(i)=0$;
  \item ${\sigma}^{\frown} 0 \notin S_{x}$ and ${\sigma}^{\frown}
  1 \in S_{x}$ if $x(i)=1$;
  \item ${\sigma}^{\frown} 0 \in S_{x}$ and ${\sigma}^{\frown} 1
  \in S_{x}$ if $x(i)=2$;
  \item ${\sigma}^{\frown} 0 \notin S_{x}$ and ${\sigma}^{\frown}
  1 \notin S_{x}$ if $x(i)=3$.
\end{itemize}
This coding can be thought of as a labeling of the nodes of
$2^{\omega}$ by the digits of $x$; a $0$ at a node means that only the 
left branch is included, a $1$ means that only the  right branch is included, 
a $2$ means that both branches are included, and a $3$ means that neither 
branch is included. Note that every tree except $2^{<\omega}$ itself has infinitely 
many codes.
% That is, given a
% binary tree $T\subseteq\str$, let $Q_T\in 4^{\omega}$ be defined as
% follows.  If $\tau$ is the $(n+1)$-st string in the
% length-lexicographical order, then
% \begin{itemize}
%   \item if $\tau0\in T$ and $\tau1\notin T$, we set $Q_T(n)=0$;
%   \item if $\tau0\notin T$ and $\tau1\in T$, we set $Q_T(n)=1$;
%   \item if $\tau0\in T$ and $\tau1\in T$, we set $Q_T(n)=2$;
%   \item if $\tau0\notin T$ and $\tau1\notin T$, we set $Q_T(n)=3$.
% \end{itemize}
% Clearly this gives a one-to-one correspondence between binary trees
% and members of $4^{\omega}$.

Let $\mu_{GW}$ be the measure on $4^{\omega}$ induced by setting, for
each $\sigma\in 4^{<\omega}$,
\begin{displaymath}
  \mu_{GW}(\sigma0\mid\sigma)=\mu_{GW}(\sigma1\mid\sigma)=2/9,\;
  \mu_{GW}(\sigma2\mid\sigma)=4/9,\;\text{and}\; \mu_{GW}(\sigma3\mid\sigma)=1/9.
\end{displaymath}

Via this coding we can also think of $\mu_{GW}$ as a measure on
$\Tree$, the space of binary trees. Then the probability of extending
a string in a tree by only $0$ is $2/9$, by only $1$ is $2/9$, by both
$0$ and $1$ is $4/9$, and by neither is $1/9$. We call a tree $T$
\emph{GW-random} if it has a random code; i.e.\ if there is $x\in
\MLR_{\mu_{GW}}$ such that $T=S_{x}$ (here $GW$ stands for Galton-Watson,
since $GW$-trees are obtained by a Galton-Watson process).

% Let $\mu_{2/3}$ be the measure on $\cs$ such that $\mu_{2/3}([\sigma
% 1])/\mu_{2/3}([\sigma]) = 2/3$ for every $\sigma \in \str$. Define a
% coding $x \mapsto S_{x}$ of subsets of $\str$ by $\cs$ via
% $\sigma_{i} \in S_{x}$ iff $x(i) = 1$. So with the measure
% $\mu_{2/3}$ on $\cs$, this says that each string has a probability
% $2/3$ of being in $S\subseteq \str$, and independently so. For
% $S\subseteq \str$, set $\tr(S)=\{\sigma \in S: \tau \preceq \sigma
% \Rightarrow \tau \in S\}$, the largest subtree of $S$. Now we call a
% tree GW-random if it is of the form $\tr(S_{x})$ for some $x \in
% \MLR_{\mu_{2/3}}$. Notice that this measure says that the
% probability that $\sigma 0 \in \tr(S)$ given that $\sigma \in
% \tr(S)$ is

% RCSs are exactly the sets of paths through infinite trees generated
% by using a real x â 4 Ï to label nodes in the same way,
% but with 3âs representing dead ends and with probabilities
% 2/9 of 0, 2/9 of 1, 4/9 of 2 and 1/9 of 3.

\begin{lemma}[Diamondstone and Kj{\o}s-Hanssen \cite{Diamondstone:2012aa}]
  A closed set $C$ is random if and only if $C$ is the set of paths
  through an infinite GW-random tree. \label{lemma:DiamKjosGWtree}
\end{lemma}

By means of Lemma \ref{lemma:DiamKjosGWtree} we prove:

\begin{theorem}\label{thm:zeros-rcf}
  \begin{enumerate}[(i)]
    \item For every random $F\in\F(\cs)$, $Z_F$ is a random closed set
    provided that it is non-empty. \label{it:zeros-rcfi}
    \item For every random $C\in\C(\cs)$, there is some random
    $F\in\F(\cs)$ such that $C=Z_F$. \label{it:zeros-rcfii}
  \end{enumerate}
\end{theorem}

\begin{proof}
  We define a computable map $\Psi\colon \mathcal{F}(\cs) \to \Tree$
  that pushes forward $\lambda_{\mathcal{F}}$ to $\mu_{GW}$ such that
  the set of paths through $\Psi(F)\cap \dom(F)$ is exactly $Z_{F}$.
  Given our representation of functions as members of $3^{\omega}$ and
  binary trees as members of $4^{\omega}$, we are really defining a
  computable map $\widehat{\Psi}\colon 3^{\omega} \to 4^{\omega}$ that
  pushes forward $\lambda$ to $\mu_{GW}$.

  Given $F\in \mathcal{F}(\cs)$, which we think of as a
  $\{0,1,2\}$-labeling of the edges of the full binary tree, we build
  the desired tree by declaring that $\sigma \in \Psi(F)$ if and only
  if the labels by $F$ of the edges of $\sigma$ consists only of $0$'s
  and $2$'s.  More formally, as in the paragraph preceding
  Definition~\ref{def:RCF-defn}, $F$ comes with a labeling function
  $\ell_{F} \colon 2^{<\omega} \setminus \{\epsilon\} \to 3$
  defined by $\ell_{F}(\sigma_{i}) = j$ if and only if $x(i)=j$ where
  $x$ is the given code for $F$. So $\sigma\in\Psi(F)$ if and
  only if $\ell_{F}(\sigma\uh k)\in\{0,2\}^{<\omega}$ for every
  $0<k\leq |\sigma|$. Clearly this map is computable.

  % \chris{We should need to introduce this notation earlier in the
  % paper: a labeling function $\ell_F$, then we could write here
  % $\sigma\in\Psi(\sigma)$ if and only if $\ell_F(\sigma\uh
  % k)\in\{0,2\}^{<\omega}$ for every $0<k\leq |\sigma|$}.

  Now we show that the map $\Psi$ pushes $\lambda_{\mathcal{F}}$
  forward to $\mu_{GW}$.  Suppose $\sigma\in\Psi(F)$, which, as
  stated above, means that $\ell_F(\sigma\uh k)\in\{0,2\}^{<\omega}$
  for every $0<k\leq|\sigma|$.  Then
  \[
  \sigma 0\in \Psi(F)\;\;\&\;\; \sigma1\notin
  \Psi(F)\;\;\Leftrightarrow\;\; \ell_F(\sigma0)\in\{0,2\}\;\;\&\;\;\ell_F(\sigma1)=1.
  \]
  The right-hand side of the equivalence occurs with probability
  $(2/3)(1/3)=2/9$.  Similarly,
  \[
  \sigma0\notin \Psi(F)\;\;\&\;\; \sigma1\in
  \Psi(F)\;\;\Leftrightarrow\;\;\ell_F(\sigma0)=1
  \;\;\&\;\;\ell_F(\sigma1)\in\{0,2\},
  \]
  where this latter event also occurs with probability $2/9$.  Next,
  \[
  \sigma0\in \Psi(F)\;\;\&\;\; \sigma1\in
  \Psi(F)\;\;\Leftrightarrow\;\; \ell_F(\sigma0)\in\{0,2\}\;\;\&\;\;\ell_F(\sigma1) \in \{0,2\},
  \]
  with the latter event occurring with probability $(2/3)(2/3)=4/9$.
  Lastly,
  \[
  \sigma0\notin \Psi(F)\;\;\&\;\; \sigma1\notin
  \Psi(F)\;\;\Leftrightarrow\;\; \ell_F(\sigma
  0)=\ell_F(\sigma1)=1,
  \]
  where the event on the right-hand side occurs with probability
  $(1/3)(1/3)=1/9$.  Now by construction, it follows immediately that
  any path through the tree $\Psi(F)$ is a sequence $X$ such that
  either $F(X)=0^{\omega}$ (in the case that $\ell_F(X\uh n)=0$ for
  infinitely many $n$) or $F(X)\diverge$ (in the case that $\ell_F(X\uh
  n)=0$ for only finitely many $n$).

  By preservation of randomness and no randomness ex nihilo, a tree is
  GW-random if and only if it is the image of some random continuous
  function $F$. The conclusion then follows by
  Lemma~\ref{lemma:DiamKjosGWtree}.
\end{proof}

% A given $\sigma\in \Psi(F)$ gets extended by only $0$ if and only if
% the edges above $\sigma$ in $F$'s tree are either $0$ and $1$ or $2$
% and $1$. This happens with probability $1/3\times 1/3 + 1/3 \times
% 1/3 = 2/9$. Similarly, the probability of extending by only $1$ is
% $2/9$. Extending by both $0$ and $1$ happens if and only if neither
% edge is labeled with a $1$, which happens with probability
% $2/3\times 2/3 = 4/9$. Extending by neither happens if and only if
% both edges are labeled with a $1$, which happens with probability
% $1/3 \times 1/3 = 1/9$.

One consequence of Theorem \ref{thm-rcfs-image-of-comp} and
Theorem~\ref{thm:zeros-rcf}(\ref{it:zeros-rcfii}), not noted in
\cite{Barmpalias:2008aa}, is that the composition of two random
continuous functions need not be random.

\begin{corollary}
  For every random $F\in\F(\cs)$, there is some random $G\in\F(\cs)$
  such that $G\circ F$ is not random.
\end{corollary}

\begin{proof}
  By Theorem \ref{thm-rcfs-image-of-comp}, there is some $R\in\MLR$
  such that $F(0^\omega)=R$. By Theorem \ref{thm-members-RCS}, there
  is some random $C \in \C(\cs)$ containing $R$. By
  Theorem~\ref{thm:zeros-rcf}(\ref{it:zeros-rcfii}), there is a $G
  \in \F(\cs)$ such that $G^{-1}(\{0^{\omega}\}) = C$.
  % Then let $G\in\F(\cs)$ be a random function with the property that
  % $G(R)=0^\omega$ \chris{need to argue this exists, but surely it
  % does}.
  It follows that $G(F(0^\omega))=0^\omega$, which implies with
  Theorem \ref{thm-rcfs-image-of-comp} that $G\circ F$ is not random.
\end{proof}

%\quinn{Question: can we get such a $G$ relatively random to $F$?}
%\chris{Good question: Clearly, $R$ cannot be $F$-random.  Can we
%  choose an $F$-random $C$ containing $R$?  If so, then $G$ would be
%  $F$-random by relativizing to $F$ the proof that every random closed
%  set is the set of zeros of a random continuous function.}
  
Another consequence of Theorem \ref{thm:zeros-rcf} is that we
can answer another open question from \cite{Barmpalias:2008aa} involving
random pseudo-distance functions.  Given a closed set $C\in\C(\cs)$, a function
$\delta:\cs\rightarrow\cs$ is a \emph{pseudo-distance function for $C$} if $C$ is
the set of zeros of $\delta$.  In \cite{Barmpalias:2008aa} it was shown that
if $\delta$ is a random pseudo-distance function for some $C\in\C(\cs)$, then
$C$ is a random closed set, but the converse was left open.  By Theorem  
\ref{thm:zeros-rcf}, the converse immediately follows.

\begin{corollary}
Let $C\in\C(\cs)$.  $C$ has a random pseudo-distance function if and only
if $C$ is a random closed set.
\end{corollary}

%  \item Does every random closed set possess a random pseudo-distance
%  function?  (Open question from Cenzer paper) \chris{Extremely
%    likely, using the previous result}

\section{The support of a random measure}\label{sec-support}

In the previous section, we established a correspondence between
random closed sets and and random continuous functions: a closed set
$C$ is random if and only if it is the set of zeros of some random
continuous function.  In this section, we establish similar correspondences
between random closed sets and random measures.

Since the support of a measure $\mu$, i.e., the
set $\Supp(\mu)=\{x\in\cs:\forall n\;\mu(x\uh n)>0\}$ is a closed set, one 
might hope to establish such a correspondence by considering the supports 
of random measures.  However, it is not hard to see that for each random 
measure $\mu$, $\Supp(\mu)=\cs$.  

If we consider a different computable measure on
$\P(\cs)$ than the measure $P$ defined above in Section \ref{subsec-random-measures},
then such a correspondence can be given.  In the first place, we want a 
measure $Q$ on $\P(\cs)$
with the property that no $Q$-random measure has full support.  In
fact, we can choose a measure $Q$ such that each $Q$-random measure is
supported on a random closed set.

%To define such a measure $Q$, we will make use of the map
%$\Theta:\cs\rightarrow\ccs$ that maps the binary representation of a
%sequence to its ternary representation, which is defined on every
%$Y\in\MLR$.  Then for $Y\in\MLR$, $\Theta(Y)$ is a Martin-L\"of random
%member of $\ccs$, and thus it can be used to define a random closed
%set as above.  Note further that for every Martin-L\"of random
%$Z\in\ccs$, $\Phi^{-1}(Z)$ is defined and is a Martin-L\"of random
%member of $\cs$.

\begin{theorem}
  There is a computable measure $Q$ on $\mathscr{P}(\cs)$ such that
  \begin{itemize}
    \item[(i)] every $Q$-random measure is supported on a random
    closed set, and
    \item[(ii)] for every random closed set $C\subseteq\cs$,
    there is a $Q$-random measure $\mu$ such that $\Supp(\mu)=C$.
  \end{itemize}
\end{theorem}

\begin{proof}

  We will define the measure $Q$ so that each $Q$-random measure is
  obtained by restricting the Lebesgue measure to a random closed set.
  That is, each $Q$-random measure will be uniform on all of the
  branching nodes of its support.

  We define $Q$ in terms of an almost total functional
  $\Phi:\ccs\rightarrow\cs$.  On input $x\in\ccs$, $\Phi$ will treat
  $x$ as the code for a closed set and will output the sequence
  $y=\oplus_{i\in\omega}y_i$ defined as follows.  For each
  $i\in\omega$, we set
  \[
  y_i= \left\{
    \begin{array}{lll}
      1^\infty & \mbox{if } x(i)=0 \\
      0^\infty & \mbox{if } x(i)=1 \\
      10^\infty & \mbox{if } x(i)=2 
    \end{array}.
  \right.
  \]
If we think of the columns of $y$ as encoding the conditional probabilities of a measure $\mu_y$, then
if $(\sigma_i)_{i\in\omega}$ is the standard enumeration of $\str$, these conditional probabilities are given
by
  \[
  p_{\sigma_i}= \left\{
    \begin{array}{lll}
      1 & \mbox{if } x(i)=0 \\
      0 & \mbox{if } x(i)=1 \\
      1/2 & \mbox{if } x(i)=2 
    \end{array}.
  \right.
  \]
  That is, $\Phi(x)=y$, where $y$ represents the unique measure  $\mu_y$ such that
  $\mu_y(\sigma0\mid\sigma)=p_\sigma$ for each $\sigma\in\str$.  Let $Q$
  be the measure on $\mathscr{P}(\cs)$ induced by the composition of $\Phi$ and the representation
  map $\Psi:\cs\rightarrow\P(\cs)$ defined in Section \ref{subsec-random-measures}.

  We now verify (i) by showing that $\Phi$ maps each $x\in\MLR$ to a
  $Q$-random measure supported on a random closed set.  Let
  $x\in\MLR$ and set $\Phi(x)=y$.  By preservation of randomness, $\Psi(\Phi(x))=\mu_y$ is $Q$-random.
 
  Next, since $x\in\MLR$, $[T_x]$ is a random closed set.
  We claim that $\Supp(\mu_y)=[T_x]$.  Suppose that
  $\sigma\in\str$ is the $(n+1)$-st extendible node of
  $T_x$.  Then one of the following holds:
  \begin{itemize}
    \item[(a)] $\sigma0\in T_x$ and $\sigma1\notin
    T_x$;
    \item[(b)] $\sigma0\notin T_x$ and $\sigma1\in
    T_x$; or
    \item[(c)] $\sigma0\in T_x$ and $\sigma1\in
    T_x$.
  \end{itemize}
  Moreover, we have
  \begin{itemize}
    \item Condition (a) holds iff $x(n)=0$ iff $\mu_y(\sigma 0\mid\sigma)=1$ and $\mu_y(\sigma
    1\mid\sigma)=0$.
    \item Condition (b) holds iff $x(n)=1$ iff $\mu_y(\sigma 0\mid\sigma)=0$ and $\mu_y(\sigma
    1\mid\sigma)=1$.
    \item Condition (c) holds iff $x(n)=2$ iff $\mu_y(\sigma 0\mid\sigma)=\mu_y(\sigma
    1\mid\sigma)=1/2$.
  \end{itemize}
  One can readily verify that $\mu_y(\sigma^\frown i\mid \sigma)>0$ if
  and only if $\sigma^\frown i\in T_x$.  Thus
  \begin{equation*}
    \begin{split}
      Z\in \Supp(\mu_y) &\Leftrightarrow \mu_y(Z\uh n)>0\;\text{for every}\;n \\
      & \Leftrightarrow \mu_y(Z\uh (n+1)\mid Z\uh n)>0\;\text{for every}\;n \\
      &\Leftrightarrow Z\uh (n+1)\in T_x\;\text{for every}\;n \\
      &\Leftrightarrow Z\in[T_x].
    \end{split}
  \end{equation*}
  We have thus established that $\mu_y$ is supported on a random closed
  set.

  To show (ii), let $C\subseteq\cs$ be a random closed set.  By definition there 
  is some Martin-L\"of random $z\in\ccs$ such that $C=[T_z]$.  Hence $\Psi(\Phi(z))$ is a
  $Q$-random measure $\nu$.  By the definition of $\Phi$, $\nu$ has
  support $[T_z]=C$, which establishes the claim.
\end{proof}

Instead of changing the measure on $\P(\cs)$ we can also establish a correspondence
between random closed sets and random measures by considering not the support of 
a random measure but what we refer to as its 1/3-support.

\begin{definition}
  Let $\mu\in\P(\cs)$ and set
  \[
  T_\mu=\{\sigma:(\forall
  i<|\sigma|)\;[\;\mu\bigl(\sigma\uh(i+i)\mid\sigma\uh
  i\bigr)>1/3\;]\}\cup\{\epsilon\}.
  \]
 Then the 1/3-\emph{support} of the measure $\mu$ is the closed set $[T_\mu]$.
\end{definition}

\begin{theorem}
  A closed set $C\in\C(\cs)$ is random if and only it is the
  1/3-support of some random measure $\mu\in\P(\cs)$.
\end{theorem}

\begin{proof}
  We define an almost-total, computable, and
  Lebesgue-measure-preserving map $\Phi:2^\omega\rightarrow 3^\omega$
  that induces a map $\tilde{\Phi}:\mathscr{P}(\cs)\rightarrow\C(\cs)$
  such that $\tilde{\Phi}(\mu)=[T_\mu]$.  Given $x=\oplus x_i\in\cs$
  such that $\mu({\sigma_i}^\frown0 \mid \sigma_{i})=x_i$ for each
  $i$, then for $\sigma\in T_\mu$ (which must exist since $\epsilon\in
  T_\mu$),
  \begin{itemize}
    \item if $\mu(\sigma 0)\in[0,1/3)$, then $\sigma1\in T_\mu$ and
    $\sigma0\notin T_\mu$;
    \item if $\mu(\sigma 0)\in(2/3,1]$, then $\sigma0\in T_\mu$ and
    $\sigma1\notin T_\mu$;
    \item if $\mu(\sigma 0)\in(1/3,2/3)$, then $\sigma0\in T_\mu$ and
    $\sigma1\in T_\mu$; and
    \item if $\mu(\sigma0)=1/3$ or $\mu(\sigma0)=2/3$, then $\Phi(x)$
    is undefined.
  \end{itemize}
  Clearly $\Phi$ is defined on a set of measure one since it is
  defined on all sequences $x$ such that $x_i\neq 1/3$ and $x_i\neq
  2/3$ for each $i$.  Observe that each $\sigma\in T_\mu$ extends to an infinite
  path in $[T_\mu]$.  Thus, if $\sigma$ is the $(n+1)$-st extendible
  node in $T_\mu$, then the each of the events
  \begin{itemize}
    \item $\sigma0\in T_\mu$ and $\sigma1\notin T_\mu$,
    \item $\sigma0\notin T_\mu$ and $\sigma1\in T_\mu$, and
    \item $\sigma0\in T_\mu$ and $\sigma1\in T_\mu$,
  \end{itemize}
  occurs with probability 1/3, since each event corresponds to whether
  $\mu(\sigma 0)\in[0,1/3)$, $\mu(\sigma 0)\in(2/3,1]$, or $\mu(\sigma
  0)\in(1/3,2/3)$, respectively.  It thus follows that the pushforward
  measure induced by $\lambda$ and $\Phi$ is the Lebesgue measure on
  $3^\omega$.  By preservation of randomness, each random measure
  $\mu$ is mapped to a random closed set, and by no randomness ex
  nihilo, each random closed set is the image of a random measure
  under $\Phi$.  This establishes the theorem.
\end{proof}

\section{The range of a random continuous function}\label{sec-range}

%Before we establish a number of results that relate random closed
%
%
%Basic facts: The image of a computable point under a random function
%is random.  \chris{preservation of randomness}
%
%\chris{Questions:}
%\begin{itemize}
%  \item If $F$ is random, is $F\circ F$?
%  \item If $F$ and $G$ are relatively random, is $F\circ G$ random?
%  \item If $F$ is random and $X\in\cs$ is $F$-random, is $F(X)\in\cs$
%  random?
%  \item Is there a random continuous function $F$ and a decidable
%  $\Pi^0_1$ class $\mathcal{P}$ such that $F(\mathcal{P})$ is a random
%  closed set?  What about if $\mathcal{P}$ is undecidable?
%
%  % \item Can a random continuous function map a random real to a
%  % computable real? \chris{Yes}
%  \item Does every random closed set possess a random pseudo-distance
%  function?  (Open question from Cenzer paper) \chris{Extremely
%    likely, using the previous result}
%  \item What is the relationship between algorithmically random
%  Brownian motion and random continuous functions? \chris{Probably not
%    necessary to answer at this point}
%  \item If $F$ is random, does $\ran(F)$ contain a computable
%  point? \quinn{By Kolmogorov's 0-1 law, almost surely it does.}
%  \item If $F$ is random and $y\in \ran(F)$, is $|F^{-1}(y)|=
%  2^{\omega}$? \quinn{We proved this when $y$ is computable or random
%    relative to $F$.}
%\end{itemize}

In \cite{Barmpalias:2008aa}, it was shown that for each $y\in\cs$
\[
\lambda(\{x\in\cs:y\in\ran(F_x)\})=3/4.
\]
from which it follows that every $y\in\cs$ is in the range of some random
$F\in\F(\cs)$.  In this section, we prove that $\lambda(\ran(F))\in(0,1)$ for 
every random function $F$.  First we will prove that $\lambda(\ran(F))>0$
for each random function, from which it follows that no random function
is injective and that the range of a random function is never a random closed
set.  These improve two results of \cite{Barmpalias:2008aa} according to which (i)
not every random function is injective and (ii) the range of a random function is not
necessarily a random closed set.  Our proof requires us to prove some auxiliary 
facts about the measure induced by a random function.

To prove that $\lambda(\ran(F))<1$ for every $F\in\F(\cs)$, we will show that
no random function is surjective, from which the result immediately follows.  Our
result on surjectivity also improves a result of \cite{Barmpalias:2008aa} according
to which not every random function is surjective.

We begin by proving the following, which is similar to a result in \cite{Culver:2014aa} 
for random measures.

\begin{lemma}
  Let $\lambda_{\F}$ be the measure on $\F(\cs)$ induced by the correspondence between
  $\F(\cs)$ and $3^\omega$. Then the  measure $P_{\F}$ on $\mathscr{P}(\cs)$ induced by the map $F\mapsto
  \lambda \circ F^{-1}$ has barycenter $\lambda$; i.e.\
  \begin{displaymath}
    \lambda(\sigma) = \int_{\mathscr{P}(\cs)} \mu(\sigma)\, dP_{\F}(\mu)
  \end{displaymath}
  for each $\sigma \in \str$. \label{lem:RCF-comp-leb-barycent}
\end{lemma}

\begin{proof}
  By change of variables, it suffices to show that
  \begin{equation}
    2^{-|\sigma|} = \int_{\F(\cs)} \lambda(F^{-1}\llb\sigma\rrb)\,
    d\lambda_{\F} \label{eq:barycenter-varchanged}
  \end{equation}
  for each $\sigma \in \str$. Without loss of generality, we assume
  $\sigma = 0^{n}$. We proceed then by induction on $n$.

  Equation~\eqref{eq:barycenter-varchanged} holds when
  $\sigma=\epsilon$ since each random $F$ is total by Theorem \ref{thm-rcfs-total}.

  Now supposing that equation~\eqref{eq:barycenter-varchanged} holds for
  $0^{n}$, we show it also holds for $0^{n+1}$. Suppose then that
  $\int_{\F(\cs)} \lambda(F^{-1}\llb0^{n}\rrb)\, d\lambda_{\F} = 2^{-n}$. To
  compute $\int_{\F(\cs)} \lambda(F^{-1}\llb0^{n+1}\rrb)\, d\lambda_{\F}$,
  we note that by symmetry $\int_{\F(\cs)} \lambda(F^{-1}\llb0^{n+1}\rrb)\,
  d\lambda_{\F} = 2\cdot \int_{\F(\cs)} \lambda(\llb0\rrb\cap
  F^{-1}\llb0^{n+1}\rrb)\, d\lambda_{\F}$ and proceed to compute $s_{n+1}
  \eqdef \int_{\F(\cs)} \lambda(\llb0\rrb\cap F^{-1}\llb0^{n+1}\rrb)\,
  d\lambda_{\F}$.

  Recall that any $F\in \F(\cs)$ can be viewed as a labeling by $0$'s,
  $1$'s, and $2$'s of the nodes of full binary branching tree (where
  the root node is unlabeled). We compute $\int_{\F(\cs)}
  \lambda(\llb0\rrb\cap F^{-1}\llb0^{n+1}\rrb)\, d\lambda_{\F}$ by considering the
  three equiprobable cases for the label of the node $0$ for an arbitrary $F\in\F(\cs)$. 
  The point is that the label $0$ contributes to producing an output beginning with 
  $0^{n+1}$, the label $1$ rules out the possibility of producing an output beginning 
  with  $0^{n+1}$, and the label $2$ neither contributes to nor rules out the possibility 
  of producing an output beginning with $0^{n+1}$.\\

  \begin{enumerate}[\text{Case} 1:]
    \item If the node $0$ is labeled with a $0$, then the measure of all
    sequences extending the node $0$ that (after removing $2$'s) yield 
    an output extending $0^{n+1}$ is equal to the measure of all sequences
    that yield an output extending $0^n$ times 1/2 (the measure determined
    by the initial label 0), i.e., $1/2\cdot 2^{-n}$.\\

    \item If the node $0$ is labeled with a $1$, then the measure of all
    sequences extending the node $0$ that (after removing $2$'s) yield 
    an output extending $0^{n+1}$ is equal to 0.\\
    
    \item If the node $0$ is labeled with a $2$, then the measure of all
    sequences extending the node $0$ that (after removing $2$'s) yield 
    an output extending $0^{n+1}$ is equal to the measure of all sequences
    that yield an output extending $0^{n+1}$ times 1/2 (the measure determined
    by the initial label 2), i.e., $1/2\cdot s_{n+1}$.\\
   
  \end{enumerate}
  
%   \begin{enumerate}[\text{Case} 1:]
%    \item If the node $0$ is labeled $0$, then the measure of all
%    strings starting with $0$ that (after removing $2$'s) yield
%    $0^{n+1}$ is half the relative measure in $[0]$ of all strings
%    above $0$ that yield $0^{n}$.
%
%    \item If the node $0$ is labeled $1$, then the measure of all
%    strings starting with $0$ that (after removing $2$'s) yield
%    $0^{n+1}$ is $0$.
%
%    \item If the node $0$ is labeled $2$, then the measure of all
%    strings starting with $0$ that (after removing $2$'s) yield
%    $0^{n+1}$ is half the relative measure in $[2]$ of all strings
%    above $0$ that yield $0^{n+1}$.
%  \end{enumerate}

  Putting this all together gives
  \begin{displaymath}
    s_{n+1} = \frac{1}{3} \cdot \frac{1}{2}\cdot 2^{-n} +
    \frac{1}{3}\cdot 0 + \frac{1}{3} \cdot \frac{1}{2} \cdot 2 s_{n+1}
  \end{displaymath}
  which yields $s_{n+1}=2^{-n}/4$, as desired.

\end{proof}

\begin{lemma}[Hoyrup \cite{Hoyrup:2013aa}, relativized]
  Let $Q$ be a computable measure on $\mathscr{P}(2^{\omega})$ with
  barycenter $\mu$. Then for any $z\in \cs$
  \begin{displaymath}
    \MLR^{z}_{\mu} = \bigcup_{\nu \in \MLR^{z}_{Q}} \MLR^{z}_{\nu}.
  \end{displaymath}
  \label{lem:barycent-rand-decom}
\end{lemma}

\begin{theorem}\label{thm:rcf-range-nonnull}
  If $F\in\F(\cs)$ is random, then $\lambda(\ran(F))>0$.
\end{theorem}

\begin{proof}
  Fix a random $F\in \F(\cs)$. We show that $\ran(F)$ always contains
  an element of $\MLR^{F}_{\lambda}$. Since $\ran(F)$ is $\Pi^{0,F}_{1}$ by Remark \ref{rmk-rcfs},
  it follows by Proposition \ref{prop-pi01-null} that $\lambda(\ran(F))>0$.
  
  By preservation of randomness relative to $F$, if $x \in
  \MLR^{F}_{\lambda}$, then $F(x) \in \MLR^{F}_{\lambda\circ F^{-1}}$.
  But by Lemmas~\ref{lem:RCF-comp-leb-barycent}
  and~\ref{lem:barycent-rand-decom}, $\MLR^{F}_{\lambda\circ F^{-1}}
  \subseteq \MLR^{F}_{\lambda}$, so $F(x) \in \MLR_{\lambda}^{F}$, as
  desired.
\end{proof}

\begin{corollary}\label{cor-rcfs-injective}
  If $F\in\F(\cs)$ is random, then $F$ is not injective.
\end{corollary}

\begin{proof}
  For any $y\in 2^{\omega}$, a relativization of
  Theorem~\ref{thm:zeros-rcf}(\ref{it:zeros-rcfi}) shows that
  $F^{-1}(\{y\})$, if non-empty, is a random closed set relative to $y$
  provided that $F$ is random relative to $y$. Since $\ran(F)$ has positive
  Lebesgue measure, there is $y\in \ran(F)$ that is random relative to $F$. But
  then by Van Lambalgen's theorem, $F$ is also relative to $y$. So
  $F^{-1}(\{y\})$ is a non-empty random closed set and hence has size
  continuum by Theorem \ref{thm-racs-perfect}. Thus $F$ is not injective.
\end{proof}

\begin{corollary}
  If $F\in\F(\cs)$ is random, then $\ran(F)$ is not a random
  closed set.
\end{corollary}

\begin{proof}
By Theorem \ref{thm-racs-null}, every random closed set
  has Lebesgue measure 0.  But by Theorem \ref{thm:rcf-range-nonnull}, the
  range of a random $F\in\F(\cs)$ has positive Lebesgue measure, and thus
  the conclusion follows.
\end{proof}

From the proof of Corollary \ref{cor-rcfs-injective} we can also obtain the following.

\begin{corollary}
Let $F\in\F(\cs)$ be random.  Then the measure $\lambda_F$ induced by $F$ is atomless, that is, $\lambda_F(\{x\})=0$ for every $x\in\cs$.
\end{corollary}

\begin{proof}
Let $F\in\F(\cs)$ be random and suppose that $z\in\cs$ is an atom of $\lambda_F$, i.e., $\lambda_F(\{z\})>0$.  It follows that $z\in\MLR_{\lambda_F}^F$, since $z$ is not contained in any $\lambda_F$-nullsets.  As we argued in the proof of Corollary \ref{cor-rcfs-injective},
$F^{-1}(\{z\})$ is a non-empty random closed set and thus has Lesbesgue measure zero by Theorem \ref{thm-racs-null}, contradicting our assumption.
\end{proof}

We now turn to showing that $\lambda(\ran(F))<1$ for every random $F\in\F(\cs)$.  Instead of proving this directly, we will first prove the following.  

\begin{theorem}\label{thm-rcf-notonto}
  If $F\in\mathcal{F}(\cs)$ is surjective, then $F$ is not random.
\end{theorem}

To prove Theorem \ref{thm-rcf-notonto}, we provide a careful analysis of the result from \cite{Barmpalias:2008aa} stated at the beginning of this section, namely that for each $y\in\cs$,
\[
\lambda(\{x\in\cs:y\in\ran(F_x)\})=3/4.
\]
This result is obtained by showing that the strictly decreasing sequence $(q_n)_{n\in\omega}$ defined by
\[
q_n=\lambda(\{x\in\cs:\ran(F_x)\cap\llb 0^n\rrb\})
\]
converges to $3/4$ and using the fact that
\[
\lambda(\{x\in\cs:\ran(F_x)\cap\llb 0^n\rrb\})=\lambda(\{x\in\cs:\ran(F_x)\cap\llb \sigma\rrb\})
\]
for each $\sigma\in\str$ of length $n$.  This sequence $(q_n)_{n\in\omega}$ is obtained by using a case analysis to derive the following recursive formula:
  \begin{equation}\label{eq-qn}
    q_{n+1}=\frac{3}{2}\sqrt{1+4q_n}-\frac{3}{2}-q_n.
  \end{equation}
For details, see \cite[Theorem 2.12]{Barmpalias:2008aa}.

For $F\in\F(\cs)$ and $\sigma\in\str$, let us say that $F$ \emph{hits}
$\llb\sigma\rrb$ if $\ran(F)\cap\llb\sigma\rrb\neq\emptyset$.  Thus, $q_n$ is the probability that a random $F\in\mathcal{F}(\cs)$
hits $\llb\sigma\rrb$ for some fixed $\sigma\in\str$ such that $|\sigma|=n$.  

%It is worth noting that the function $\mathcal{T}(\llb\sigma\rrb)=q_n$ for each $\sigma$ of length $n$ induces an effective capacity on $\C(\cs)$; see \cite{Brodhead:2011aa} for details on effective capacities.

We will proceed by proving a series of lemmas.  First, for each $n\in\omega$, let $\epsilon_n$ satisfy $q_n=3/4+\epsilon_n$.  Since 
  \begin{itemize}
    \item[(i)] $q_n>q_{n+1}$ for every $n$, and
    \item[(ii)] $\lim_{n\rightarrow\infty}q_n=3/4$.
  \end{itemize}
 we know that each $\epsilon_n$ is non-negative and $\lim_{n\rightarrow\infty}\epsilon_n=0$.  Moreover, we have the following.
  
\begin{lemma}\label{lem-rcf-notonto1}
  For each $n\geq 1$,
  \begin{itemize}
    \item[(a)] $\epsilon_{n+1}\leq\frac{1}{2}\epsilon_n$,
    \item[(b)] $\epsilon_n\leq 2^{-(n+2)}$,
    \item[(c)] $\epsilon_{n+1}\geq\frac{1}{2}\epsilon_n-2^{-(2n+5)}$,
    and
    \item[(d)] $\epsilon_n\geq\frac{1}{2^{n+5}-1}$.
  \end{itemize}
\end{lemma}

\begin{proof}
First, let $n\geq 1$.  If we substitute $3/4+\epsilon_{n+1}$ and $3/4+\epsilon_n$ for $q_{n+1}$ and $q_n$, respectively, into Equation (\ref{eq-qn}), we obtain (after simplification)
  \begin{equation}\label{eq-epsilon}
    \epsilon_{n+1}=3\sqrt{1+\epsilon_n}-3-\epsilon_n.
  \end{equation}
  Since $\sqrt{1+x}\leq 1+\frac{x}{2}$ on $[0,1]$, from (\ref{eq-epsilon}) we
  can conclude
  \[
  \epsilon_{n+1}\leq
  3\bigl(1+\frac{\epsilon_n}{2}\bigr)-3-\epsilon_n=\frac{1}{2}\epsilon_n,
  \]
  thereby establishing (a).  To show (b), we proceed by induction.
  Using the fact from \cite{Barmpalias:2008aa} that $q_1=\frac{\sqrt{45}-5}{2}$,
  it follows by direct calculation that
  \[
  \epsilon_1=\frac{\sqrt{45}-5}{2}-\frac{3}{4}\leq 2^{-3}.
  \]
  Next, assuming that $\epsilon_n\leq 2^{-(n+2)}$, it follows from (a)
  that
  \[
  \epsilon_{n+1}\leq\frac{1}{2}\epsilon_n\leq
  \frac{1}{2}2^{-(n+2)}=2^{-(n+3)}.
  \]
  To show (c), for each fixed $n\geq 1$, we use a different
  approximation of $\sqrt{1+x}$ from below.  By (b), since
  $\epsilon_n\leq 2^{-(n+2)}$, we use the Taylor series approximation
  $1+\frac{x}{2}$ of $\sqrt{1+x}$ centered at 0 on $[0,2^{-(n+2)}]$ with error term
  \[
  \max_{c\in[0,2^{-(n+2)}]}\dfrac{1}{4(1+c)^{3/2}}\dfrac{x^2}{2}=\dfrac{x^2}{8}.
  \]
  Thus,
  \[
  \sqrt{1+x}\geq 1+\frac{x}{2}-\bigl(2^{-(n+2)}\bigr)^2/8=1+\frac{x}{2}-2^{-(2n+7)}
  \]
  on $[0,2^{-(n+2)}]$.  Combining this with Equation (\ref{eq-epsilon})
  yields
  \[
  \epsilon_{n+1}\geq
  3(1+\frac{\epsilon_n}{2}-2^{-(2n+7)})-3-\epsilon_n\geq
  \frac{1}{2}\epsilon_n-2^{-(2n+5)}.
  \]
  Lastly, to prove (d), first observe that
  \begin{equation}\label{eq-ep1}
    \epsilon_1=\frac{\sqrt{45}-5}{2}-\frac{3}{4}\geq 2^{-4}
  \end{equation}
  and thus it certainly follows that
  \[
  \epsilon_1\geq \dfrac{1}{2^6-1}.
  \]
  Next, using (c), we verify by induction that for $n\geq 2$,
  \begin{equation}\label{eq-moreepsilon}
    \epsilon_n\geq\frac{1}{2^{n-1}}\epsilon_1-\bigl(2^{-(n+5)}+\dotsc+2^{-(2n+3)}\bigr).
  \end{equation}
  For $n=2$, by part (c) we have
  \[
  \epsilon_2\geq\frac{1}{2}\epsilon_1-2^{-7}.
  \]
  Supposing that
  \[
  \epsilon_n\geq\frac{1}{2^{n-1}}\epsilon_1-\bigl(2^{-(n+5)}+\dotsc+2^{-(2n+3)}\bigr),
  \]  
 again by part (c) we have
 \begin{equation*}
 \begin{split}
  \epsilon_{n+1}\geq  \frac{1}{2}\epsilon_n-2^{-(2n+5)}&\geq\frac{1}{2}\Bigl(\frac{1}{2^{n-1}}\epsilon_1-\bigl(2^{-(n+5)}+\dotsc+2^{-(2n+3)}\bigr)\Bigr)-2^{-(2n+5)}\\
  &=\frac{1}{2^n}\epsilon_1-\bigl(2^{-(n+6)}+\dotsc+2^{-(2n+4)}\bigr)-2^{-(2n+5)}\\
  &=\frac{1}{2^n}\epsilon_1-\bigl(2^{-(n+6)}+\dotsc+2^{-(2n+5)}\bigr),
  \end{split}
 \end{equation*}
which establishes Equation (\ref{eq-moreepsilon}).
Combining Equations (\ref{eq-ep1}) and (\ref{eq-moreepsilon}) yields
  \begin{equation*}
  \begin{split}
  \epsilon_n&\geq\frac{1}{2^{n-1}}2^{-4}-\bigl(2^{-(n+5)}+\dotsc+2^{-(2n+3)}\bigr).\\
  &= \frac{1}{2^{(n+3)}}-2^{-(n+4)}\bigl(2^{-1}+\dotsc+2^{-(n-1)}\bigr)\\
  &= \frac{1}{2^{(n+3)}}-2^{-(n+4)}(1-2^{-(n-1)})\\
  &\geq 2^{-(n+3)}-2^{-(n+4)}\\
  &\geq  2^{-(n+4)}\\
  &\geq \dfrac{1}{2^{n+5}-1}.
  \end{split}
  \end{equation*}

\end{proof}

\begin{lemma}\label{lem-rcf-notonto2}
  For $n\geq 1$, we have
  \[
  \dfrac{q_{n+1}}{q_n}\leq 1-2^{-(n+6)}.
  \]
\end{lemma}

\begin{proof}

  By Lemma \ref{lem-rcf-notonto1}(d),
  \[
  \epsilon_n\geq\dfrac{1}{2^{n+5}-1}=\dfrac{2^{-(n+5)}}{1-2^{-(n+5)}},
  \]
  which implies
  \[
  \bigl(1-2^{-(n+5)}\bigr)\epsilon_n\geq 2^{-(n+5)}=4\cdot2^{-(n+7)}\geq
  3\cdot2^{-(n+7)}.
  \]
  Multiplying both sides by 1/2 yields
  \[
  \frac{1}{2}\bigl(1-2^{-(n+5)}\bigr)\epsilon_n\geq\frac{3}{4}2^{-(n+6)}.
  \]
  Expanding the left hand side and using the fact from Lemma \ref{lem-rcf-notonto1}(a) that
  $\frac{1}{2}\epsilon_n\geq\epsilon_{n+1}$, we have
  \[
  \frac{1}{2}\epsilon_n+\Bigl(\frac{1}{4}+\dotsc+2^{-(n+6)}\Bigr)\epsilon_n\geq\frac{3}{4}2^{-(n+6)}+\epsilon_{n+1}
  \]
  which is equivalent to
  \[
  (1-2^{-(n+6)})\epsilon_n+\frac{3}{4}(1-2^{-(n+6)})\geq\frac{3}{4}+\epsilon_{n+1}
  \]
  This yields the inequality
  \[
  \bigl(1-2^{-(n+6)}\bigr)q_n\geq q_{n+1},
  \]
  from which the conclusion follows.
\end{proof}

\begin{lemma}\label{lem-rcf-notonto3}
  For $n\geq 1$, we have
  \[ {\Biggl(2\Bigl(\dfrac{q_{n+1}}{q_n}\Bigr)-1\Biggr)}^{2^n}\leq
  \frac{1}{\sqrt[32]{e}}<1.
  \]
\end{lemma}

\begin{proof}
  First, it follows from Lemma \ref{lem-rcf-notonto2} that
  \[
  2\Bigl(\dfrac{q_{n+1}}{q_n}\Bigr)-1\leq 1-2^{-(n+5)}
  \]
  and hence
  \begin{equation}\label{eq5}
    {\Biggl(2\Bigl(\dfrac{q_{n+1}}{q_n}\Bigr)-1\Biggr)}^{2^n}\leq
    \Bigl(1-2^{-(n+5)}\Bigr)^{2^n}.
  \end{equation}
  Next, it is straightforward to verify by cross-multiplication that
  \[
  \dfrac{2^{n+5}-1}{2^{n+5}}\leq\dfrac{2^{n+6}-1}{2^{n+6}}
  \]
  and
  \[
  \dfrac{2^{n+5}-1}{2^{n+5}}\leq\Biggl(\dfrac{2^{n+6}-1}{2^{n+6}}\Biggr)^2,
  \]
  from which it follows that
  \[
  {\Biggl(\dfrac{2^{n+5}-1}{2^{n+5}}\Biggr)}^{2^n}\leq\Biggl(\dfrac{2^{n+6}-1}{2^{n+6}}\Biggr)^{2^{n+1}}.
  \]
  Lastly, we have
  \[
  \lim_{n\rightarrow\infty}\Bigl(1-2^{-(n+5)}\Bigr)^{2^n}=\dfrac{1}{\sqrt[32]{e}}.
  \]
  From Equation (\ref{eq5}) and the fact that the sequence
  $\bigl((1-2^{-(n+5)})^{2^n}\bigr)_{n\in\omega}$ is non-decreasing and
  converges to $1/\sqrt[32]{e}$, the claim immediately follows.
\end{proof}

The proof of following result is essentially the proof of the
effective Choquet capacity theorem in \cite{Brodhead:2011aa}.  We
reproduce the proof here for the sake of completeness.

\begin{lemma}
  \label{lem-rcf-notonto3}
  The probability that a random continuous function $F$ hits both
  $\llb0\rrb$ and $\llb1\rrb$ is $2q_1-1$, and the probability that
  $F$ hits both $\llb\sigma0\rrb$ and $\llb\sigma1\rrb$ for a fixed
  $\sigma\in\str$ of length $n\geq 1$, given that $F$ hits
  $\llb\sigma\rrb$, is equal to $2\Bigl(\dfrac{q_{n+1}}{q_n}\Bigr)-1$.
\end{lemma}

\begin{proof}
  First, let us write the probability that $F$ hits $\llb\sigma\rrb$
  for some fixed $\sigma$ as $\mathbb{P}(F\in H_\sigma)$.  Now since
  $\mathbb{P}(F\in H_0)=q_1$, it follows that $\mathbb{P}(F\in
  H_1\setminus H_0)=1-q_1$ (here we use the fact that every random function
  is total).  By symmetry, we have $\mathbb{P}(F\in
  H_0\setminus H_1)=1-q_1$.  Since $F$ is total with probability one,
  it follows that
  \[
  \mathbb{P}(F\in H_0\cap H_1)=1-\bigl(\mathbb{P}(F\in H_0\setminus
  H_1)+\mathbb{P}(F\in H_1\setminus H_0)\bigr)
  \]
  and thus
  \[
  \mathbb{P}(F\in H_0\cap H_1)=1-((1-q_1)+(1-q_1))=2q_1-1.
  \]
  Next, let $\sigma$ be a string of length $n\geq 1$ and let
  $i\in\{0,1\}$.  Since $\mathbb{P}(F\in H_\sigma)=q_n$ and
  $\mathbb{P}(F\in H_{\sigma^\frown i})=q_{n+1}$ it follows that
  \[
  \mathbb{P}(F\in H_{\sigma^\frown i}\mid F\in
  H_\sigma)=\dfrac{\mathbb{P}(F\in
    H_{\sigma^\frown i}\;\&\;F\in H_{\sigma})}{\mathbb{P}(F\in
    H_\sigma)}=\dfrac{\mathbb{P}(F\in H_{\sigma^\frown
      i})}{\mathbb{P}(F\in H_\sigma)}=\dfrac{q_{n+1}}{q_n}.
  \]
  Consequently,
  \[
  \mathbb{P}(F\in H_{\sigma1}\setminus H_{\sigma0}\mid F\in H_\sigma)=\mathbb{P}(F\in
  H_{\sigma0}\setminus H_{\sigma1}\mid F\in H_\sigma)=1-\dfrac{q_{n+1}}{q_n}
  \]
  Thus,
  \begin{equation*}
    \begin{split}
      \mathbb{P}(F\in H_{\sigma0}\cap H_{\sigma1}\mid F\in H_\sigma )&=1-\bigl(\mathbb{P}(F\in H_{\sigma0}\setminus H_{\sigma1}\mid F\in H_\sigma)+\mathbb{P}(F\in H_{\sigma1}\setminus H_{\sigma0}\mid F\in H_\sigma)\bigr)\\
      &=1-\Bigl(\bigl(1-\dfrac{q_{n+1}}{q_n}\bigr)+\bigl(1-\dfrac{q_{n+1}}{q_n}\bigr)\Bigr)\\
      &=2\Bigl(\dfrac{q_{n+1}}{q_n}\Bigr)-1.
    \end{split}
  \end{equation*}
\end{proof}

To complete the proof of Theorem \ref{thm-rcf-notonto}, we now define
a Martin-L\"of test on $\F(\cs)$ that covers all surjective functions.  Let
us say that a function $F\in\F(\cs)$ is \emph{onto up to level $n$} if
$F\in H_\sigma$ for every $\sigma\in2^n$.  By Lemma
\ref{lem-rcf-notonto3}, the probability of a function being onto up to
level $n$ is
\[
(2q_1-1)\prod_{i=1}^{n-1}{\Biggl(2\Bigl(\dfrac{q_{i+1}}{q_i}\Bigr)-1\Biggr)}^{2^i}\leq
{\Biggl(\dfrac{1}{\sqrt[32]{e}}\Biggr)}^{n}.
\]
Thus, if we set
\[
\U_n=\{F\in\F(\cs):F\;\mathrm{is\;onto\;up\;to\;level}\;n\},
\]
and
\[
f(n)=\min\{k:(\sqrt[32]{e})^{-k}\leq 2^{-n}\},
\]
which is clearly computable, 
then $(\U_{f(n)})_{n\in\omega}$ is a Martin-L\"of test with the property
that $F\in\F(\cs)$ is onto if and only if
$F\in\bigcap_{n\in\omega}\U_{f(n)}$. This completes the proof.

\begin{corollary}\label{cor:rcf-range-nonfull}
  If $F\in\F(\cs)$ is random, then $\lambda(\ran(F))<1$.
\end{corollary}

\begin{proof}
Suppose $\lambda(\ran(F))=1$.  Then since $\ran(F)$ is closed, it follows that $\ran(F)=\cs$.  But then $F$ is onto, so it cannot be random.  
\end{proof}

% New proof that every random closed set contains a Martin-L\"of
% random sequence, which immediately yields the converse that every
% Martin-L\"of random sequence is contained in some random closed set.

We also have the following corollary.

\begin{theorem}
No measure induced by a random function is a random measure in the sense of Definition \ref{defn-random-measure}.
\end{theorem}

\begin{proof}
Let $F\in\F(\cs)$ be random.  Then by Corollary \ref{cor:rcf-range-nonfull}, $\lambda(\ran(F))<1$.  Thus, it follows that $\cs\setminus\ran(F)$ is non-empty and open, so  $\llb\sigma\rrb\subseteq\cs\setminus\ran(F)$ for some $\sigma\in\str$.  Thus, $\lambda_F(\sigma)=0$.  By contrast, for every random measure $\mu$, we have $\mu(\sigma)>0$, and the result follows.
\end{proof}

%We conclude with an open question about the map $F\mapsto \lambda(\ran(F))$.  It is not hard to see that this map is a.e.\ computable.  Indeed, since each random function $F$ is total, for each $n$ there is some level $k$ of the $\ell_F$-labelled tree such that every length $k$ path is labelled with at least $n$ 0s or 1s, which allows us to 

%\appendix
%\section{}\label{appendix}
%
%
%In this appendix, we establish the following lemmas used in the proof
%of Theorem \chris{???}.  Recall that $q_n=3/4+\epsilon_n$ is the
%probability that $\ran(F)\cap\llb\sigma\rrb\neq\emptyset$ for a fixed
%$\sigma$ of length $n$.
%
%\begin{lemma}
%  For each $n\geq 1$,
%  \begin{itemize}
%    \item[(a)] $\epsilon_{n+1}\leq\frac{1}{2}\epsilon_n$,
%    \item[(b)] $\epsilon_n\leq 2^{-(n+2)}$,
%    \item[(c)] $\epsilon_{n+1}\geq\frac{1}{2}\epsilon_n-2^{-(2n+5)}$,
%    and
%        \item[(d)] $\epsilon_n\geq\frac{1}{2^{n+5}-1}$.
%  \end{itemize}
%\end{lemma}

\bibliographystyle{alpha} \bibliography{interplay}

\end{document}